\documentclass[a4paper, 11pt]{style/preprint}
\usepackage[left=1.5in, right=1.5in, bottom=1.5in]{geometry}
\usepackage{style/comment}
\usepackage{style/mhenvs}
\usepackage{style/mhsymb}

\usepackage[full]{textcomp}
\usepackage[osf]{newtxtext}
\usepackage{amssymb}
\usepackage{amsmath}
\usepackage{mathtools}
\usepackage{mhequ}
\usepackage{booktabs}
\usepackage{tikz}
\usepackage{mathrsfs}
\usepackage{longtable}
\usepackage{microtype}
\usepackage{wasysym}
\usepackage{centernot}
\usepackage{enumitem}
\usepackage[dvipsnames]{xcolor}
\usepackage{hyperref}
\usepackage{crossreftools}
\usepackage{physics}
\usepackage{biblatex}
\usepackage{orcidlink}

\mathtoolsset{showonlyrefs=true}

\newif\ifdark
\IfFileExists{dark}{\darktrue}{\darkfalse}

\ifdark

\definecolor{darkred}{rgb}{0.9,0.2,0.2}
\definecolor{darkblue}{rgb}{0.7,0.3,1}
\definecolor{darkgreen}{rgb}{0.1,0.9,0.1}
\definecolor{pagebackground}{rgb}{.15,.21,.18}
\definecolor{pageforeground}{rgb}{.84,.84,.85}
\pagecolor{pagebackground}
\AtBeginDocument{\globalcolor{pageforeground}}

\else

\definecolor{darkred}{rgb}{0.7,0.1,0.1}
\definecolor{darkblue}{rgb}{0.4,0.1,0.8}
\definecolor{darkgreen}{rgb}{0.1,0.7,0.1}
\definecolor{pagebackground}{rgb}{1,1,1}
\definecolor{pageforeground}{rgb}{0,0,0}

\fi

\definecolor{newred}{RGB}{208,16,76}
\definecolor{newgreen}{RGB}{34,125,81}

\makeatletter
\newcommand{\globalcolor}[1]{%
	\color{#1}\global\let\default@color\current@color
}
\makeatother

\DeclareSymbolFont{timesoperators}{T1}{ptm}{m}{n}
\SetSymbolFont{timesoperators}{bold}{T1}{ptm}{b}{n}

\makeatletter
\renewcommand{\operator@font}{\mathgroup\symtimesoperators}
\makeatother

\DeclareMathAlphabet{\mathbbm}{U}{bbm}{m}{n}
\DeclareFontFamily{U}{BOONDOX-calo}{\skewchar\font=45 }
\DeclareFontShape{U}{BOONDOX-calo}{m}{n}{
	<-> s*[1.05] BOONDOX-r-calo}{}
\DeclareFontShape{U}{BOONDOX-calo}{b}{n}{
	<-> s*[1.05] BOONDOX-b-calo}{}
\DeclareMathAlphabet{\mcb}{U}{BOONDOX-calo}{m}{n}
\SetMathAlphabet{\mcb}{bold}{U}{BOONDOX-calo}{b}{n}

\makeatletter 

\newcommand*{\fat}{}
\DeclareRobustCommand*{\fat}{%
	\mathbin{\mathpalette\bigcdot@{}}}
\newcommand*{\bigcdot@scalefactor}{.5}
\newcommand*{\bigcdot@widthfactor}{1.15}
\newcommand*{\bigcdot@}[2]{%
	\sbox0{$#1\vcenter{}$}
	\sbox2{$#1\cdot\m@th$}%
	\hbox to \bigcdot@widthfactor\wd2{%
		\hfil
		\raise\ht0\hbox{%
			\scalebox{\bigcdot@scalefactor}{%
				\lower\ht0\hbox{$#1\bullet\m@th$}%
			}%
		}%
		\hfil
	}%
}

\DeclareRobustCommand{\TitleEquation}[2]{\texorpdfstring{\StrLeft{\f@series}{1}[\@firstchar]$\if b\@firstchar\boldsymbol{#1}\else#1\fi$}{#2}}

\makeatother

\theoremnumbering{Alph}
\renewtheorem{theorem*}{Theorem}

\numberwithin{equation}{section}

\def\slash{\leavevmode\unskip\kern0.18em/\penalty\exhyphenpenalty\kern0.18em}
\def\dash{\leavevmode\unskip\kern0.18em--\penalty\exhyphenpenalty\kern0.18em}

\let\epsilon\varepsilon

\def\${|\!|\!|}

\setlist{noitemsep,topsep=4pt}
\def\para_#1{/\!\!/_{\!#1}}

\begin{document}

\title{Central limit theorem for the Allen-Cahn equation 
with supercritical random initial conditions}

\author{Colin Piernot \orcidlink{0009-0007-4448-5328} and
  Kexing Ying \orcidlink{0000-0002-8292-4746}}

\institute{Institute of Mathematics, EPFL, Switzerland}

\maketitle

\begingroup
\renewcommand\thefootnote{}
\footnote{The authors acknowledge support from NCCR SwissMAP.}
\addtocounter{footnote}{-1}
\endgroup

 \begin{abstract}

 We study the large-scale behavior of solutions to the Allen-Cahn reaction-diffusion equation 
 with Gaussian initial data.
 We consider the case of short-range dependence in the associated supercritical regime with spatial dimension $d \ge 3$. In this case, the non-linearity formally vanishes on large scales under the diffusive rescaling.
 Accordingly, we prove a central limit theorem for the rescaled solution, more precisely, that it converges to the solution of the heat equation started from a white noise.
 These initial conditions for the limit depend non-trivially both on the source of randomness and on the non-linearity. Our proof uses estimates obtained by a combination of comparison principles and Malliavin calculus, initiated in \cite{dunlapcastillo} in the critical case.
 However, the result in \cite{dunlapcastillo} is not a fluctuation result but rather an $L^2_\mathbf{P}$ comparison to a McKean-Vlasov problem with Gaussian solutions.
 Hence the mechanism behind the Gaussianity of the limit differs, and the proof requires new ideas that should be further applicable to other supercritical problems.

 
\end{abstract}

\section{Introduction}

The aim of this article is to study the large scale behavior of the Allen-Cahn equation
\begin{equation}\label{eq:AC}
  \partial_t u(t,x) = \Delta u(t,x) - \lambda u(t,x) ^3,\quad t>0,\:x\in \mathbb{R}^d,
\end{equation}
where the coupling parameter $\lambda$ is assumed to be positive, in dimension $d\geq 3$. We consider the case of random initial conditions:
the field $\big(u(0,x)\big)_{x\in \mathbb{R}^d}$ is a smooth,
centered and stationary Gaussian field with short-range correlations. More precisely, we assume
 that $u(0,\cdot) = \rho * \xi$ , where $\xi$ is a white noise on $\mathbb{R}^d$ and $\rho$ is a
smooth, non-negative function with compact support such that $\int_{\mathbb{R}^d} \rho(x) \dd x = 1$. In particular,
the covariance function of the initial condition, given by
$\mathbf{E}[u(0,0)u(0,x)] = C_{\text{init}}(x) = \rho * \tilde{\rho}(x)$, where $\tilde{\rho}(x) = \rho(-x)$, is integrable. 
In this framework, the study of the large scale behavior of the solution $u$ is of interest since the 
non-linearity is \textit{supercritical} when $d \geq 3$. Indeed, na\"ively rescaling diffusively the 
solution to \eqref{eq:AC}, i.e. defining for $\eps > 0$,
\begin{equation}\label{eq:rescaled_u}
  u_\eps(t,x) = \eps^{-\frac{d}{2}} u(\eps^{-2} t, \eps^{-1} x),
\end{equation}
we see that $u_\eps$ solves
\begin{equation}\label{eq:rescaled_AC}
  \partial_t u_\eps(t,x) = \Delta u_\eps(t,x) - \lambda \eps^{d-2} u_\eps(t,x)^3 ,
\end{equation}
with initial condition 
\[u_\eps(0,x) = \eps^{-\frac{d}{2}} u(0, \eps^{-1} x) = \eps^{-\frac{d}{2}} (\rho * \xi)(\eps^{-1} x) 
  \overset{\text{law}}{=} \rho_\eps * \xi \underset{\eps\downarrow 0}{\rightarrow} \xi .\]
This seems to suggest that, at large scales, the non-linear term vanishes and the solution behaves 
like the solution to the heat equation with white noise initial condition. However, crucially, the 
pathwise solution map $ u_\eps(0) \mapsto u_\eps $ cannot be constructed as a continuous map in a 
reasonable topology to justify this argument.

This is in contrast with the \textit{subcritical case} of Gaussian initial data of regularity 
$\alpha > - 1$ where one can actually construct an appropriate solution map on a space modelled 
after the noise: see for example the works \cite{chevyrev,HairerRosati,Cann4} that replicate part of 
the so-called pathwise solution theory for singular SPDEs of \cite{Hairer} and \cite{Gub}.
The impossibility of carrying this out here is because the equation is \textit{supercritical} in 
the sense of singular SPDEs.

The goal of stepping away from the subcritical setting, by tackling supercritical and critical 
equations has raised much attention in recent years. Substantial progress has been made in several 
settings, however heavily relying on specific features of the problems considered. To give just some examples, a
line of works \cite{Cann1, Cann2, Cann3} consider the case of stationary solutions in both the
supercritical and critical settings, to equations with Gaussian invariant measures such as the AKPZ
equation and the stochastic Burgers equation. Other works such as \cite{Nikos, Gu, HairerGerolla1, HairerGerolla2} study the
stochastic heat equation, but rely heavily on either explicit representations of the solution, or on
a martingale structure (or the KPZ equation, via a Cole-Hopf transform).

Due to the supercriticality and the lack of explicit representations, understanding the large scale behavior of the solution to \eqref{eq:AC} is a
non-trivial task. In fact, it turns out that the non-linearity, though vanishing in the large scale limit 
\eqref{eq:rescaled_AC}, still has an effect on the limiting statistics of the solution.
Namely, our main result Theorem \ref{thm:Main} states that the solution $u_\eps$ converges in law to the 
solution of the heat equation with initial data which is white-in-space while having a modified variance. 
The convergence result obtained also shows that the effective noise in the limit has non-trivial correlations with the large scales of the initial data.
This is in contrast with what was observed at the critical dimension $d=2$ in \cite{gabrielrosati} and  
\cite{dunlapcastillo}, in the so called weak-coupling regime. In these works, the authors study a similar problem but in the critical setting, where further taming of the non-linearity is required to 
prove convergence. They proved as expected that the non-linearity affects the limiting variance, but 
there the correlations with the initial data are trivial. Indeed, they show $L^2_\mathbf{P}$ convergence of the solution to 
an explicit function of the microscopic noise $\xi$. In \cite{gabrielrosati}, this is carried out through 
a careful investigation of a combinatorial expansion of the solution, reminiscent of the methods of the 
subcritical regime. On the other hand, in \cite{dunlapcastillo}, an $L^2_\mathbf{P}$ comparison to a McKean-Vlasov problem, 
via tools of Malliavin calculus, is used and allows to treat more general non-linearities. Crucially, 
both arguments rely heavily on the monotonicity property of the non-linearity, a feature that remains at 
the very core of our argument. However, the actual mechanism behind the limiting result, as explained
in detail in Section \ref{sec:outline}, turns out to be different. This explains both the 
difference in the notions of convergence, as well as the absence of an explicit representation for the 
limiting variance. 

\subsection{Main results}
In this section, we state precisely our main result concerning the large scale behavior of the solution to \eqref{eq:AC}.
\begin{theorem}\label{thm:Main}
  Let $d\geq 3$ and $u$ be the solution to \eqref{eq:AC} with initial condition as above for some $\lambda >0$.
  For $\eps > 0$, let $u_\eps$ be the rescaled solution defined in \eqref{eq:rescaled_u}.
  Then, as $\eps \downarrow 0$, the family of processes $(u_\eps)_{\eps > 0}$ converges in law in
  $C\left((0,+\infty), \mathcal{S}'(\mathbb{R}^d)\right)$ to a Gaussian process that solves the heat 
  equation with initial condition $\sigma_\lambda \tilde{\xi}$, where $\tilde{\xi}$ is spatial white noise, 
  and $\sigma_\lambda$ is a positive constant. Namely,
  $u_\eps \Rightarrow( \sigma_\lambda p_t*\tilde{\xi})_{t>0} $. 

  Furthermore, we have that
  \begin{enumerate}[label=(\roman*)]
    \item\label{item:nontriv} $\sigma_\lambda$ depends on $\lambda$ non-trivially. In particular, 
      $\lim_{\lambda \to + \infty}\sigma_\lambda = 0$.
    \item\label{item:creation} for $\lambda$ sufficiently small, 
      $\left(u_\eps(t), p_{t} * u_\eps(0) \right)_{t > 0}$ converges in law to a Gaussian pair of 
      processes with non-trivial correlations. In particular, 
      $$\left(u_\eps(t), p_{t} * u_\eps(0)\right)_{t>0} \Rightarrow( \sigma_\lambda p_{t}*\tilde{\xi},  p_{t}*\xi)_{t>0} $$
      where $\xi,\tilde{\xi}$ are standard white noises with $$\mathbf{E}[\xi(x)\tilde{\xi}(y)] = c_\lambda \delta(x-y)\text{ for some }0<c_\lambda<1.$$
  \end{enumerate}
\end{theorem}

\begin{remark}
Note that the convergence in law result is for positive time only, because by definition $u_\eps$ at time $0$ is given.
Also, we note the following implicit representation holds for the limiting variance
$$\sigma^2_\lambda = \lim_{t \to \infty } \int_{\mathbb{R}^d} \mathbf{E}[u(t,x)u(t,0)]\dd x,$$
where $u$ is the unscaled solution of Equation~\eqref{eq:AC}. We are not aware of a more direct representation.
\end{remark}
We refer to \ref{item:nontriv} and \ref{item:creation} respectively as non-triviality of the variance 
and creation of noise. This is explained in more detail in Section \ref{subsec:nontriv}.

Let us point out that aside from the monotonicity and oddness of the non-linearity, most of the 
arguments used in the proof of Theorem \ref{thm:Main} are quite robust. They should apply to a larger 
class of problems at supercriticality without too much effort. In particular, it should be straightforward 
to extend the result to the case of initial data arising from a Poisson point process, or to a 
discretized version of the problem with i.i.d. initial data, as long as the law of the initial data 
is centered and symmetric. We note that the method here also extends directly to the case of the non-linearity $-\lambda u^{2k + 1}$ for some $k \geq 1$, in the appropriate
supercritical dimensions.

Finally, the case of long-range initial condition is also of interest. The case of integrable correlations
is a straightforward extension of the result. We believe a similar result should hold in the case of non-integrable correlations
with a supercritical power-law decay, where the limiting effective initial data would be a fractional Gaussian field. 

Of course, it would also be interesting to explore related directions of work that go beyond the 
maximum principle. For example, a similar result for the Cahn-Hilliard equation or the unforced 
AKPZ equation would be very interesting. Another problem of interest would be to replicate the 
reasoning here in cases of equations with additive noise. For example, the $\Phi^4$ equation in 
dimension $d\geq 5$ is supercritical, and one expects a similar result to hold. However, the presence 
of renormalization makes the use of a maximum principle delicate and the problem much more challenging.

\subsection{Notations and conventions}

We write $\mathbf{E}$ for the expectation with respect to the probability measure $\mathbf{P}$, and $\mathbb{E}_x$ for the expectation with respect to a Brownian motion started at $x$.
We denote by $\mathcal{H} = L^2(\mathbb{R}^d)$ the Hilbert space associated to the white noise $\xi$, with inner product $\langle \cdot, \cdot \rangle_{\mathcal{H}}$.

We denote by $p_t(x) = (4\pi t)^{-d/2} e^{-|x|^2/(4t)}$ the heat kernel, and by $p_t * f$ the convolution with $f$.
The mollifier $\rho$ is a smooth, non-negative, compactly supported function with $\int_{\mathbb{R}^d} \rho(x) \dd x = 1$, and we set $\tilde{\rho}(x) = \rho(-x)$. For $\eps > 0$, we define
\[\rho_\eps(x) = \frac{1}{\eps^d}\rho\left(\frac{x}{\eps}\right), \quad \lambda_\eps = \lambda \eps^{d - 2}.\]

We write $\mathcal{S}'(\mathbb{R}^d)$ for the space of tempered distributions.
The Malliavin derivative is denoted by $D$ (or $D_y$ when evaluated at $y \in \mathbb{R}^d$), and $\mathbb{D}^{1,p}$ denotes the associated Sobolev space. We denote by $\mathcal{H}_n$ the $n$-th homogeneous Wiener chaos and by $\Pi_n$ the orthogonal projection onto $\mathcal{H}_n$.

Finally, we use the notation $f \lesssim g$ (resp.\ $f \gtrsim g$) to indicate that there exists a constant $C > 0$, depending only on the dimension $d$ and the mollifier $\rho$, such that $f \leq C g$ (resp.\ $f \geq C g$). Dependence on additional parameters is indicated by subscripts, e.g.\ $\lesssim_p$. We write $f \asymp g$ when both $f \lesssim g$ and $f \gtrsim g$ hold.

\section{Outline of the proof}\label{sec:outline}
In this section, we give an overview of the main steps of the proof of Theorem \ref{thm:Main}.
\subsection{Heuristics behind the mechanism of convergence}
To argue heuristically why the solution to \eqref{eq:AC} converges to a Gaussian field at large 
scales, let us start by writing the mild formulation of the equation, starting from some time $t \in (0, T)$
\begin{equation}\label{eq:mild}
  u(T,x) = p_{T-t} * u(t,\cdot)(x) - \lambda \int_t^T (p_{T-s} * u(s,\cdot)^3)(x) \dd s,\ x \in \mathbb{R}^d,
\end{equation}
and let us make the natural assumption that at time $t$,  $u(t,\cdot)$ is correlated on length scales 
of order $O(\sqrt{t})$, and is of order $O(t^{-\frac{d}{4}})$ pointwise --- which is the typical size 
suggested by the naïve scaling. In particular, this implies that the cubic non-linearity $u(t,\cdot)^3$ 
is of order $O(t^{-\frac{3d}{4}})$ pointwise, and is correlated on length-scales of order $O(\sqrt{t})$ as well.
This suggests that the variance of the non-linear part in \eqref{eq:mild} can be estimated as follows
\begin{align}
  & \mathbf{E}\left[\left( p_{T-s} * u(s,\cdot)^3(x)\right)^2\right] \\
  =\ & \int_{\mathbb{R}^d} \int_{\mathbb{R}^d} p_{T-s}(x-y) p_{T-s}(x-z) \mathbf{E}\left(u(s,y)^3 u(s,z)^3\right) \dd y \dd z \\
  \approx\ & \int_{\mathbb{R}^d} \int_{\mathbb{R}^d} p_{T-s}(x-y) p_{T-s}(x-z) s^{-d} p_s(y-z) \dd y \dd z 
    \sim s^{-d} T^{-\frac{d}{2}},
\end{align}
where we used the approximation $\mathbf{E}\left(u(s,y)^3 u(s,z)^3\right) \approx s^{-\frac{3d}{2}} 
\mathbf{1}_{|y-z|\lesssim \sqrt{s}} \approx s^{-d} p_s(y-z)$ following the assumptions on the size and 
correlation length, as well as the fact that $u(t,x)^3$ is centered. 

Thus, the contribution in $L^2_\mathbf{P}$ of the non-linear part of the dynamics between time $t$ 
and $T$ to $u(T,x)$ is of order
\begin{equation}
  T^{-d/4} \int_t^T O(s^{-d/2}) \dd s = T^{-d/4} \times 
  \begin{cases}
    O(\log(T/t)) & \text{ if } d=2,\\
    O(t^{1-d/2}) & \text{ if } d\geq 3.
  \end{cases}
\end{equation}
In the case $d \geq 3$, this suggests that the non-linear part is initially of the same order as the 
linear part, but as the time $t$ gets larger, the non-linear part has a contribution that gets smaller 
and smaller. Hence, at time $T \gg t \gg 1$, the solution $u(T,x)$ is well-approximated in 
$L^2_\mathbf{P}$ by the linear part alone, which is a spatial average of $u(t,\cdot)$ on large scales. 
At this step, the stationarity in space of the field $u(t,\cdot)$ would allow us to invoke a central 
limit theorem for spatial averages to conclude that $u(T,x)$ is approximately Gaussian, provided 
good mixing properties of $u(t,\cdot)$. This discussion further applies to several time points $T_1, \ldots, T_n \gg t $, thus hints at the convergence of the finite-dimensional distributions to those of a Gaussian process.
As one may notice, this heuristic relies mildly 
on the specifics of the problem. In particular, the vanishing of the non-linear contribution at large 
times corresponds precisely to supercriticality. The key to make this rigorous lies in proving that 
the assumptions made on the size of the field $u(t,\cdot)$ and its correlation properties at time 
$t>0$ hold. Namely, one needs to show that during the time where the non-linearity is relevant, the 
non-linearity hurts neither the size nor the decay of correlations of the field. 
This is the main bottleneck that prevents direct application of the argument to other supercritical 
problems. 

\subsection{Non-triviality of the variance and the correlations}\label{subsec:nontriv}
Let us now briefly explain the mechanism behind the non-triviality of the variance and the correlations. Both properties stem from the fact that the non-linear part of the dynamics, though vanishing 
at large times, is still relevant for times $t = O(1)$. Since the non-linearity affects what happens 
at those short times, it has a lasting effect on the large scale limit. Hence a simplified picture 
of the problem would be to brutally approximate the discretized dynamics by the following.
Since $u(0,\cdot)$ is smooth and correlated on length-scales of order $O(1)$, we may assume that for 
times of order $O(1)$ that are not too large, the diffusion has not had time to act yet, and thus the 
dynamics is dominated by the non-linear part. Thus, in this time layer, we may approximate the 
dynamics by an ODE flow at each point in space, namely
\begin{equation}
  \partial_t u_\text{ODE}(t,x) = -\lambda u_\text{ODE}(t,x)^3, \quad u_\text{ODE}(0,x) = u(0,x).
\end{equation}
This flow is explicit and we have
\begin{equation}
  u_\text{ODE}(1,x) = \Phi_\lambda\left(u(0,x)\right), \text{ where } \Phi_\lambda(u) = \frac{u}{\sqrt{1 + 2\lambda u^2}}.
\end{equation}
If one now pretends that after this time layer, the approximation of the dynamics by the heat equation is exact, we get that at time $T \gg 1$,
\begin{equation}
  u(T,x) \approx p_{T-1} * u_\text{ODE}(1,\cdot)(x) = \left(p_{T-1} * \Phi_\lambda\left(u(0,\cdot)\right)\right)(x),
\end{equation}
for which it is clear that both non-triviality of the variance and creation of noise hold. Indeed $\Phi_\lambda$ is a non-linear function and genuinely depends on $\lambda$.
Thus, having in mind the standard Central Limit Theorem for sums of i.i.d. random variables, the joint convergence of 
$p_{T} * u_\eps(0,\cdot)(x)$ with $u_\eps(T,x)$ to a Gaussian vector with non-trivial correlations is reminiscent to that of
\begin{equation}
  \left(\frac{1}{\sqrt{n}} \sum_{i=1}^n X_i, \frac{1}{\sqrt{n}} \sum_{i=1}^n  \Phi_\lambda(X_i) \right) \underset{N\rightarrow +\infty}{\Longrightarrow} \mathcal{N}\left(0, \begin{pmatrix}
    \mathbf{E}[X^2] & \mathbf{E}[X \Phi_\lambda(X)] \\
    \mathbf{E}[X\Phi_\lambda(X)] & \mathbf{E}[\Phi_\lambda(X)^2]
  \end{pmatrix}\right).
\end{equation}
Of course, this picture is formal, but still it captures the main mechanism behind these two properties. 
Let us also mention that for the critical case, the scaling invariance of the equation prevents such 
a time layer to appear, at least not to leading order in the weak-coupling. In that case, the 
non-linearity and the diffusion really are interacting. The diffusion part creates Gaussianity,
to leading order in the weak coupling right away, but cannot erase the effect of the non-linearity 
on the size of the solution. We will further comment on the critical case in Subsection \ref{subsec:non_triv-proof}.

Unfortunately, the proofs we found for these two properties are not based on this intuition, but 
rather on more technical arguments. A caveat is that the creation of noise is only rigorously derived for small 
enough $\lambda$ even though from the above we expect that it should hold for all $\lambda > 0$.
Namely, we show that for $\lambda$ small enough, the projection on the third Wiener chaos of the solution does not vanish in the limit.
Let us note that the method we use may be adapted to prove that for any $k \geq 1$, the projection on the $2k+1$ chaos is non-vanishing in the limit, provided $\lambda$ is small enough (depending on $k$).
In this small-coupling regime, this thus shows a really ``truly chaotic'' nature of the limit.

\subsection{Main tools}
The first ingredient to make the above heuristics rigorous is the use of the maximum principle 
combined with Poincaré-type inequalities for functionals of Gaussian fields. This idea was introduced 
in \cite{dunlapcastillo} to study the weakly-coupled critical Allen-Cahn equation, and relies on the 
simple observation that the Malliavin derivative solves the linearized equation around the solution.
Deferring to Subsection \ref{subsec:malliavin} for more details on the Malliavin calculus setting, 
let us already state the following equation for the Malliavin derivative of the solution to \eqref{eq:AC}.
\begin{proposition}\label{prop:malliavin_derivative}
  Let $u$ be the solution to \eqref{eq:AC} with initial condition as above. Then, for any $t>0$ and 
  $x\in \mathbb{R}^d$, $u(t,x)$ is Malliavin differentiable. Furthermore, for all $y\in \mathbb{R}^d$ 
  the Malliavin derivative $D_{y} u$ solves
  \begin{equation}
    \partial_t D_{y} u(t,x) = \Delta D_{y} u(t,x) - 3\lambda u(t,x)^2 D_{y} u(t,x), \quad D_{y} u(0,x) = \rho(x-y).
  \end{equation}
  At the level of the rescaled solution $u_\eps$ defined in \eqref{eq:rescaled_u}, we have for any 
  $\eps > 0$, $t>0$ and $x,y\in \mathbb{R}^d$,
  \begin{equation}\label{eq:malliavin_derivative_rescaled}
    \begin{split}
      \partial_t D_{y} u_\eps(t,x) & = \Delta D_{y} u_\eps(t,x) - 3\lambda \eps^{d-2} u_\eps(t,x)^2 D_{y} u_\eps(t,x),\\ \quad D_{y} u_\eps(0,x) & = \rho_\eps(x-y).
    \end{split}
  \end{equation}
\end{proposition}
From this equation, one gets immediately from the maximum principle (crucially, from the positivity 
of $u_\eps(t,x)^2$) that the following holds.
\begin{corollary}\label{cor:max_principle}
  Under the same notations, we have for any $\eps > 0$, $t>0$ and $x,y\in \mathbb{R}^d$, the estimate
  \begin{equation}
    0\leq D_{y} u_\eps(t,x) \leq p_t * \rho_\eps (x-y).
  \end{equation}
  Consequently, for any $p\geq1$, we get the following bounds on the $p$-th moment of $Du(t,x)$ and 
  $\langle Du(t,\cdot),\phi\rangle$ where $\phi$ is any test function.
  \begin{align*}
    \mathbf{E}\left[ \|D u_\eps(t,x)\|_{\mathcal{H}}^p\right] & \leq \|p_t * \rho_\eps\|_{L^2(\mathbb{R}^d)}^p,\\ 
    \mathbf{E}\left[ \|\langle D u_\eps(t,\cdot),\phi\rangle\|_\mathcal{H}^p\right] 
    & \leq \|p_t * \rho_\eps * |\phi|\|_{L^2(\mathbb{R}^d)}^p.
  \end{align*}
\end{corollary}

Moreover, as a direct consequence of the fact that the initial condition is stationary and centered, 
we note that the field is stationary and symmetric in space at all times, i.e. for any $t > 0, x\in \mathbb{R}^d$, we have that 
$$u_\eps(t,\cdot-x) \overset{\text{law}}{=} u_\eps(t,\cdot) \overset{\text{law}}{=} - u_\eps(t,\cdot).$$
In Section \ref{subsec:estimates}, we combine the above observations with functional inequalities on the Wiener space providing us with an explicit control on the size and correlation 
properties of the field $u(t,\cdot)$. This part of the argument is quite robust with respect to the 
type of noise in the initial data, as long as appropriate functional inequalities are available.

The second ingredient is a Central Limit Theorem for stationary fields with short-range correlations.
We use here a version adapted to the specific decorrelation estimates we obtain from the first 
ingredient, and which is stated in Appendix \ref{app:clt}. Our proof is based on the moment method.

\section{Preliminaries}
In this section, we gather several preliminary results towards the proof of Theorem \ref{thm:Main}, 
which we leave to the further Section \ref{sec:proof}.

\subsection{The Malliavin calculus setting}\label{subsec:malliavin}
Let us here introduce the setting of Malliavin calculus on the Gaussian space generated by the initial 
condition of \eqref{eq:AC}, namely, the precise framework in which Proposition 
\ref{prop:malliavin_derivative} and Corollary \ref{cor:max_principle} hold.
Let $\xi$ be a white noise on $\mathbb{R}^d$, defined on some probability space 
$(\Omega, \mathcal{F}, \mathbf{P})$. We denote by $\mathcal{H} = L^2(\mathbb{R}^d)$ the Hilbert space 
associated to $\xi$, and by $W = \{W(h): h\in \mathcal{H}\}$ the isonormal Gaussian process associated 
to $\xi$, i.e. the centered Gaussian family such that $\mathbf{E}[W(h) W(g)] = \langle h,g\rangle_{\mathcal{H}}$ 
for all $h,g\in \mathcal{H}$. In particular, we have that $\xi$ can be identified with $W$ through 
the relation $\xi(\phi) = W(\phi)$ for all $\phi \in \mathcal{H}$.

For any smooth and cylindrical random variables $F$ (i.e. random variables of the form
$F = f(W(h_1), \ldots, W(h_n))$, for $n\geq 1$, $h_1, \ldots, h_n \in \mathcal{H}$ and $f\in C_b^\infty(\mathbb{R}^n)$),
the Malliavin derivative of $F$ is the $\mathcal{H}$-valued random variable defined as
\begin{equation}
  D F = \sum_{i=1}^n \partial_i f(W(h_1), \ldots, W(h_n)) h_i.
\end{equation}
The operator $D$ is closable from $L^p(\Omega)$ to $L^p(\Omega; \mathcal{H})$ for any $p\geq 1$, and 
we denote by $\mathbb{D}^{1,p}$ the closure of smooth cylindrical random variables with respect to the norm
\begin{equation}
  \|F\|_{1,p} = \left( \mathbf{E}[|F|^p] + \mathbf{E}[\|D F\|_{\mathcal{H}}^p] \right)^{\frac{1}{p}}.
\end{equation}
The space $L^2(\Omega)$ admits the following orthogonal decomposition, known as the Wiener chaos decomposition
\begin{equation}
  L^2(\Omega) = \bigoplus_{n=0}^\infty \mathcal{H}_n,
\end{equation}
into the so-called homogeneous Wiener chaoses $\mathcal{H}_n$, which are the closed linear subspaces 
of $L^2(\Omega)$ generated by the random variables of the form $H_n(W(h))$ for some $h\in \mathcal{H}$ 
with $\|h\|_{\mathcal{H}} = 1$, where $H_n$ is the $n$-th Hermite polynomial. 
In particular, we have that $\mathcal{H}_0 = \mathbb{R}$ and $\mathcal{H}_1 = \{W(h): h\in \mathcal{H}\}$.
For $n\geq 0$ we denote by $\Pi_n$ the orthogonal projection onto $\mathcal{H}_n$.

We refer to the standard monograph \cite{Nualart} for further details. Aside from the basic definitions 
stated above and the product rule, we will use the following functional inequalities on the Wiener space.
In particular, we use the following Poincaré-type inequality which follows directly from Meyer’s 
inequalities, a reference for which can be found in \cite{Nualart}.
\begin{proposition}[$L^p$-Poincaré inequality]\label{prop:poincare}
Let $p>1$. There exists a constant $C_p<\infty$ such that, for every centered $F\in\mathbb D^{1,p}$,
$$\|F\|_{L^p(\Omega)} \le C_p \|DF\|_{L^p(\Omega;\mathcal H)}.$$
\end{proposition}
We will also need the associated covariance inequality which follows directly from the Helffer-Sjöstrand 
representation of the covariance. A reference for which can be found in \cite{duerinck-nonperturbative}.
\begin{proposition}[Covariance inequality]\label{prop:cov}
  For any $F,G \in \mathbb{D}^{1,2}$ centered, we have
  \begin{equation}
    |\mathbf{E}[FG]| \leq \int_{\mathbb{R}^d} \sqrt{\mathbf{E}[|D_x F|^2]} \sqrt{\mathbf{E}[|D_x G|^2]} \dd x.
  \end{equation}
\end{proposition}
With the above inequalities in mind, Proposition \ref{prop:malliavin_derivative} and Corollary 
\ref{cor:max_principle} follow from the results in Appendix~\ref{app:det}, together with the chain rule for Malliavin derivatives.
Note that this requires sub-exponential growth of mollified white noise, which is standard.

\subsection{Estimates on the rescaled solutions}\label{subsec:estimates}

Utilizing the tools provided by Malliavin calculus, we provide some necessary estimates on the rescaled solution $u_\eps$
as defined in \eqref{eq:rescaled_u}.

As a consequence of Corollary~\ref{cor:max_principle} and the Poincaré Inequality~\ref{prop:poincare},
we have trivially the upper bound
\begin{equation}\label{eq:upper-Lp-est}
  \mathbf{E}[\langle u_\eps(t, \cdot), \phi\rangle^p] 
    \lesssim \|p_t * \rho_\eps * |\phi|\|_{L^2(\mathbb{R}^d)}^p.
\end{equation}
We show that the reverse inequality also holds for \(p = 2\).

\begin{lemma}\label{lem:non-degen}
  For any \(\lambda > 0\), any positive test function \(\phi\) and \(t > 0\), we have that
  \[\mathbf{E}[\<u_\eps(t, \cdot), \phi\>^2] \asymp \|p_t * \rho_\eps * |\phi|\|_{L^2(\mathbb{R}^d)}^2.\]
\end{lemma}
\begin{proof}
  Projecting \(u_\eps(t, x)\)
  onto the first Wiener chaos, we have by the It\^o isometry that
  \begin{equation}\label{eq:malliavin-lower}
    \begin{split}
      \mathbf{E}[\<u_\eps(t, \cdot), \phi\>^2]
        & \ge \mathbf{E}[\<\Pi_1 u_\eps(t, \cdot), \phi\>^2]\\
        & = \int \left(\int \mathbf{E}[D_y u_\eps(t, x)]\phi(x) \dd x\right)^2 \dd y.
    \end{split}
  \end{equation}
  Thus, as \(D_y u_\eps(t, x)\) is positive, it suffices to estimate \(\mathbf{E}[D_y u_\eps(t, x)]\)
  from below.
  
  To do so, by applying the Feynman-Kac formula on \eqref{eq:malliavin_derivative_rescaled}, we obtain
  \[\mathbf{E}[D_y u_\eps(t, x)]
   = \mathbf{E}\mathbb{E}_x\left[\exp\left(- 3 \lambda \eps^{d - 2}
    \int_0^t u_\eps(s, B_{t - s})^2 \dd s\right) \rho_\eps(y - B_t)\right]\]
  with \(B\) being a Brownian motion starting from \(x\) and \(\mathbb{E}_x\) denoting the
  expectation with respect to \(B\). Thus, applying Jensen's inequality, we have by stationarity that
  \begin{equation}\label{eq:malliavin-lower'}
    \mathbf{E}[D_y u_\eps(t, x)] \ge p_t * \rho_\eps(y - x)
      \exp\left(- 3 \lambda \eps^{d - 2} \int_0^t \mathbf{E}[u_\eps(s, 0)^2] \dd s \right).
  \end{equation}

  Now, as \(\mathbf{E}[u_\eps(s, 0)^2] \lesssim \|p_s * \rho_\eps\|^2_{L^2(\mathbb{R}^d)}\)
  by the Poincaré Inequality~\ref{prop:poincare} and Corollary~\ref{cor:max_principle}, we have by
  Young's convolution inequality that 
  \[\mathbf{E}[u_\eps(s, 0)^2]
  \lesssim \|p_s\|_2^2 \|\rho_\eps\|_1^2 \wedge \|p_s\|_1^2 \|\rho_\eps\|_2^2 \lesssim 
    s^{-\frac{d}{2}} \wedge \eps^{-d} \lesssim (s+\eps^2)^{-\frac{d}{2}}.\]
  As $d\geq 3$, we have that \(\eps^{d - 2} \int_0^t \mathbf{E}[u_\eps(s, 0)^2] \dd s
    \lesssim \eps^{d - 2} \int_0^t (s+\eps^2)^{-\frac{d}{2}} \dd s \lesssim 1\)
  from which \eqref{eq:malliavin-lower'} provides the lower bound
  \[\mathbf{E}[D_y u_\eps(t, x)] \gtrsim p_t * \rho_\eps(y - x).\]
  Hence, substituting this back into \eqref{eq:malliavin-lower} concludes the proof.
\end{proof}
Let us also prove here the following lemma, that controls the spatial decay of correlations of multi-point evaluations of $u_\eps$.
\begin{lemma}\label{lem:npoint}
  For any \(t > \eps^2\) and  any \(x_1, \cdots, x_{p},y_1,\cdots,y_q \in \mathbb{R}^d\), we have that
  \begin{align}
  &\left|\mathbf{E}\left[\prod_{i=1}^p u_\eps(t, x_i) \prod_{k=1}^q u_\eps(t, y_k)\right]-\mathbf{E}\left[\prod_{i=1}^p u_\eps(t, x_i)\right]\mathbf{E}\left[\prod_{k=1}^q u_\eps(t, y_k)\right]\right|\\
  &\lesssim_{p,q} t^{-\frac{(p+q)d}{4}} \exp\left(-\frac{\underset{i,k}{\inf}\|x_i - y_k\|^2}{Ct}\right). \end{align}
  \end{lemma}
\begin{proof}
The terms of which we want to evaluate the covariance are in $ \mathbb{D}^{1,2}$, hence we can apply the Covariance Inequality \ref{prop:cov} to reduce the problem to bounding
\[ \int \sqrt{\mathbf{E}\left[\left|D_x\prod_{i=1}^p u_\eps(t, x_i)\right|^2\right]}\sqrt{\mathbf{E}\left[\left|D_x\prod_{k=1}^q u_\eps(t, y_k)\right|^2\right]}\dd x .\]
Now note that the product rule for Malliavin derivatives, Equation~\eqref{eq:upper-Lp-est}, together with the Cauchy-Schwarz inequality yields
\[\mathbf{E}[|D_x(u_\eps(t, x_1) \cdots u_\eps(t, x_{p}))|^2] \lesssim_{p} t^{-\frac{(p-1)d}{2}} \sum_{i=1}^p (p_t * \rho_\eps(x-x_i))^2,\]
as well as a similar bound for the second term. Injecting this back up in the above, together with Chapman-Kolmogorov, concludes the proof of the bound.
\end{proof}

\begin{remark}
Note that by scaling, the exact same bound holds for $u(t,x)$ whenever $t>1$. In particular, for any $t>1$, $u(t,\cdot)$ satisfies the assumptions of Theorem~\ref{thm:clt} from Appendix~\ref{app:clt}.
\end{remark}

To argue for the non-triviality, we will make use of the following coming-down from infinity property for this equation.

\begin{lemma}\label{lem:coming-down}
  One has that for any  \(t > 0\) ,
  \begin{equation}\label{eq:descent}
    \mathbf{E}[u_\eps(t,0)^2] \lesssim \frac{1}{\lambda_\eps t}.
  \end{equation}
\end{lemma}
\begin{proof}
  Taking \(w : \mathbb{R}^d \to \mathbb{R}\) a smooth non-negative function which satisfies
  \(\int w(x) \dd x = 1\), we denote \(L^2_w\) the space of square integrable functions with
  respect to the measure \(w(x) \dd x\). We have by integration by parts that
  {\hfuzz=20pt
  \begin{align*}
    \frac{1}{2}\frac{ \dd }{\dd t} \|u_\eps(t, \cdot)\|_{L_w^2}^2
      & = \<u_\eps(t, \cdot), \Delta u_\eps(t, \cdot)\>_{L_w^2} - \lambda_\eps \int u_\eps(t,\cdot)^4 \dd w\\
      & = - \|\nabla u_\eps(t, \cdot)\|_{L_w^2}^2 - \int u_\eps(t, x) \nabla u_\eps(t, \cdot) \cdot \nabla w(x) \dd x
          - \lambda_\eps \int u_\eps(t, \cdot)^4 \dd w\\
      & \le - \frac{1}{2}\<\nabla u_\eps(t, \cdot)^2, \nabla w\>_{L^2} - \lambda_\eps \int u_\eps(t, \cdot )^4 \dd w.
  \end{align*}}
  Thus, taking expectation, we have by stationarity that
  \begin{align*}
    \frac{1}{2}\frac{ \dd }{\dd t} \mathbf{E}[u_\eps(t, 0)^2] & \le
      - \frac{1}{2}\<\nabla \mathbf{E}[u_\eps(t, 0)^2], \nabla w\>_{L^2} - \lambda_\eps \mathbf{E}[u_\eps(t, 0)^4] \int w(x) \dd x\\
      & = - \lambda_\eps \mathbf{E}[u_\eps(t, 0)^4] \le - \lambda_\eps \mathbf{E}[u_\eps(t, 0)^2]^2.
  \end{align*}
  Consequently, we obtain \eqref{eq:descent} by ODE comparison.
\end{proof}

\section{Proof of Theorem~\ref{thm:Main}}\label{sec:proof}
In this section we combine the preliminary results derived above to conclude Theorem~\ref{thm:Main}.
\subsection{Proof of the convergence}

We begin by proving the Central Limit Theorem part of the result. In particular, by first showing the convergence of \(u_\eps\) 
in law when tested, we can conclude the convergence part of Theorem~\ref{thm:Main} by a simple 
application of Mitoma's criterion \cite{Mitoma1983} thanks to Lemma~\ref{lem:non-degen}.

\begin{lemma}
  Let $d\geq 3$ and defining $u_\eps$ as in \eqref{eq:rescaled_u}, we have that 
  $$\left(\<u_\eps(t_1, \cdot), \phi_1\>, \dots, \<u_\eps(t_n, \cdot), \phi_n\>\right) \underset{\eps \to 0}{\implies} \left(\<\sigma_\lambda p_{t_1} * \tilde{\xi}, \phi_1\>, \dots, \<\sigma_\lambda p_{t_n} * \tilde{\xi}, \phi_n\>\right)$$
  for any \(\phi_1, \dots, \phi_n \in \mathcal{S}\) and \(t_1, \dots, t_n > 0\) where $\tilde{\xi}$ is a spatial white noise and $\sigma_\lambda$ is a positive constant. 
\end{lemma}
\begin{proof}
  For each \(i = 1, \dots, n\), we compare \(\<u_\eps(t_i, \cdot), \phi_i\>\) and 
  \(\<u_\eps(s\eps^2, \cdot), p_{t_i - s\eps^2} * \phi_i\>\) from which we show that the latter 
  converges to the corresponding Gaussian random variable. 
  
  In particular, for \(t > 0\) and \(\phi \in \mathcal{S}\), by writing $u_\eps$ in mild form, we have that 
  \begin{align*}
    & \left\|\<u_\eps(t, \cdot), \phi\> - \<u_\eps(s\eps^2, \cdot), p_{t - s\eps^2} * \phi\>\right\|_{L^2(\mathbf{P})}\\
    =\ & \lambda \eps^{d - 2} \mathbf{E}\left[\left|\int_{s\eps^2}^{t} \<u_\eps(r, \cdot)^3, p_{t - r} *\phi\> \dd r\right|^2\right]^{\frac{1}{2}}\\
    \le\ & \lambda \eps^{d - 2} \int_{s\eps^2}^{t} \mathbf{E}[|\<u_\eps(r, \cdot)^3, p_{t - r} *\phi\>|^2]^{\frac{1}{2}} \dd r
  \end{align*}
  and we will estimate the integrand on the right hand side. Applying the Poincaré inequality and the point-wise
  estimate for the Malliavin derivative (Corollary~\ref{cor:max_principle}), we have that
  \begin{align*}
    & \mathbf{E}[|\<u_\eps(r, \cdot)^3, p_{t - r} *\phi\>|^2] 
      \le \mathbf{E}[\|D\<u_\eps(r, \cdot)^3, p_{t - r} *\phi\>\|_{\CH}^2]\\
    \le\ & 9 \int \mathbf{E}[u_\eps(r, x_1)^2 u_\eps(r, x_2)^2] \\
    & \hspace{0.5cm} (p_r * \rho_\eps)(x_1 - y) (p_r * \rho_\eps)(x_2 - y) (p_{t - r} * |\phi|)(x_1) (p_{t - r} * |\phi|)(x_2) \dd x_1 \dd x_2 \dd y\\
    \le\ & 9 \mathbf{E}[u_\eps(r, 0)^4] \|p_r * \rho_\eps * p_{t - r} * |\phi|\|^2_{L^2(\mathbb{R}^d)}
    \lesssim \|p_r * \rho_\eps\|^4_{L^2(\mathbb{R}^d)} \|p_{t} * \rho_\eps * |\phi|\|^2_{L^2(\mathbb{R}^d)}.
  \end{align*}
  Thus, substituting this back, we obtain
  \begin{equation}\label{eq:mild-estimate}
    \begin{split}
      & \left\|\<u_\eps(t, \cdot), \phi\> - \<u_\eps(s\eps^2, \cdot), p_{t - s\eps^2} * \phi\>\right\|_{L^2(\mathbf{P})} \\
      \lesssim\ & \lambda \eps^{d - 2} \|p_{t} * \rho_\eps * |\phi|\|_{L^2(\mathbb{R}^d)} \int_{s\eps^2}^{t} \|p_r * \rho_\eps\|^2_{L^2(\mathbb{R}^d)} \dd r\\
      \lesssim\ & \lambda \eps^{d - 2} \|p_{t} * \rho_\eps * |\phi|\|_{L^2(\mathbb{R}^d)} \int_{s\eps^2}^{t} (\eps^2 + r)^{-\frac{d}{2}} \dd r
    \end{split}
  \end{equation}
  for which the right hand side is bounded by 
  \(\frac{2\lambda}{d - 2}\|p_{t} * |\phi|\|_{L^2(\mathbb{R}^d)} (s + 1)^{-\frac{d}{2} + 1}\).
  
  On the other hand, as a consequence of the Central Limit Theorem~\ref{thm:clt} and the 
  estimate Lemma~\ref{lem:npoint}, we have that 
  \begin{equation}\label{eq:clt-estimate}
    \left(\<u_\eps(s\eps^2, \cdot), p_{t_i - s\eps^2} * \phi_i\>\right)_{i = 1}^n \underset{\eps\to 0}{\Longrightarrow} 
    \mathcal{N}(0, \Sigma_{s, t}^2(\Phi))
  \end{equation}
  where we introduced the positive semi-definite matrix \(\Sigma_{s, t}^2(\Phi) \in \mathbb{R}^{n \times n}\) with 
  coordinates given by
  \[\Sigma_{s, t}^2(\Phi)_{i, j} := \<p_{t_i + t_j} * \phi_i, \phi_j\> \int \mathbf{E}[u(s, 0) u(s, x)] \dd x.\]
  Moreover, utilizing the Covariance Inequality~\ref{prop:cov}, we have that
  \begin{align*}
    \sup_{s}\ \Sigma_{s, t}^2(\Phi)_{i, j}
    & \le \left|\<p_{t_i + t_j} * \phi_i, \phi_j\>\right| \int \int \sqrt{\mathbf{E}[|D_z u(s, 0)|^2]} \sqrt{\mathbf{E}[|D_z u(s, x)|^2]} \dd z \dd x\\
    & \le \left|\<p_{t_i + t_j} * \phi_i, \phi_j\>\right| \int \int (p_s * \rho)(z) (p_s * \rho)(z - x) \dd z \dd x\\[1.5mm]
    & = \left|\<p_{t_i + t_j} * \phi_i, \phi_j\>\right| < \infty.
  \end{align*}
  Thus, we may extract a subsequence \(s_n \to \infty\) such that the limit \(\lim_{n \to \infty} \Sigma_{s_n, t}^2(\Phi)\) 
  exists and we denote it by \(\Sigma_t^2(\Phi)\). Thus, denoting \(d_P\) for the L\'evy-Prokhorov metric, 
  we have that
  \begin{align}
    & d_P((\<u_\eps(t_i, \cdot), \phi_i\>)_{i = 1}^n, \mathcal{N}(0, \Sigma_t^2(\Phi)))\\
    \le\ & d_P((\<u_\eps(t_i, \cdot), \phi_i\>)_{i = 1}^n, (\<u_\eps(s_n\eps^2, \cdot), p_{t_i - s_n\eps^2} * \phi_i\>)_{i = 1}^n) \label{eq:est-term1}\\
    \quad\ & + d_P((\<u_\eps(s_n\eps^2, \cdot), p_{t_i - s_n\eps^2} * \phi_i\>)_{i = 1}^n, \mathcal{N}(0, \Sigma_{s_n, t}^2(\Phi)))\label{eq:est-term2}\\
    \quad\ & + d_P(\mathcal{N}(0, \Sigma_{s_n, t}^2(\Phi)), \mathcal{N}(0, \Sigma_t^2(\Phi))).\label{eq:est-term3}
  \end{align}
  Now, by \eqref{eq:mild-estimate} and standard estimates on Gaussian measures (utilizing the relation between 
  the L\'evy-Prokhorov metric and the Wasserstein metric), the first and last term on the 
  right hand side are bounded uniformly in \(\eps\) by 
  \[\eqref{eq:est-term1} \lesssim (s_n + 1)^{-\frac{d}{2} + 1} \text{ and }
    \eqref{eq:est-term3} \lesssim \sup_{i, j}\left|\sqrt{\Sigma_{s_n, t}^2(\Phi)_{i, j}} - \sqrt{\Sigma_t^2(\Phi)_{i, j}}\right|\]
  respectively. Thus, we conclude that 
  \begin{align*}
    & \limsup_{\eps \to 0} d_P((\<u_\eps(t_i, \cdot), \phi_i\>)_{i = 1}^n, \mathcal{N}(0, \Sigma_t^2(\Phi))) \\
    \lesssim\ & (s_n + 1)^{-\frac{d}{2} + 1}  
        + \sup_{i, j}\left|\sqrt{\Sigma_{s_n, t}^2(\Phi)_{i, j}} - \sqrt{\Sigma_t^2(\Phi)_{i, j}}\right|
  \end{align*}
  for which the right hand side vanishes as \(n \to \infty\).
\end{proof}

We remark that, as a consequence of the above convergence, the constant \(\Sigma_t^2(\Phi)\) is 
independent of the choice of the subsequence \(s_n\) (as alluded to by the notation). 

\begin{remark}\label{rmk:pointwise}
  Note that the f.d.d. convergence is stated for test functions as we will obtain tightness in a space 
  of distributions. Nonetheless, by following the same proof replacing the test function with pointwise 
  evaluation, the convergence remains to hold in f.d.d.
\end{remark}

For convergence in $C((0,\infty),\mathcal S'(\mathbb R^d))$, it remains to justify tightness in time. 
For this, by Mitoma's criterion \cite{Mitoma1983}, it is enough to prove tightness of the 
real-valued processes $(\langle u_\varepsilon(t),\phi\rangle)_{t\in[\tau,T]}$, for every 
$\phi\in\mathcal S(\mathbb R^d)$ and every $0<\tau<T<\infty$. Fix such $\phi,\tau,T$. Using the 
mild formulation, for $\tau\le s<t\le T$, we decompose
$$ \langle u_\varepsilon(t) - u_\varepsilon(s), \phi\rangle =
\langle u_\varepsilon(s),p_{t-s}*\phi-\phi\rangle-\lambda\varepsilon^{d-2}\int_s^t
\langle u_\varepsilon(r)^3,p_{t-r}*\phi\rangle\dd r .$$
For the first term, the Poincaré estimate and Corollary~\ref{cor:max_principle} give
$$\mathbf{E}\left|\langle u_\varepsilon(s), p_{t-s}*\phi-\phi\rangle\right|^2 \lesssim
\|p_s*\rho_\varepsilon*(p_{t-s}*\phi-\phi)\|_{L^2}^2.$$
Since $s\ge\tau$, the heat semigroup is smoothing uniformly in $\varepsilon$, and hence, for every $\alpha\in(0,1]$,
$$\|p_s*\rho_\varepsilon*(p_{t-s}*\phi-\phi)\|_{L^2}
\lesssim_{\tau,T,\phi,\alpha}|t-s|^\alpha.$$
For the nonlinear term, the estimate used in the proof of Lemma 4.1 yields, uniformly for $r\in[\tau,T]$,
$$\mathbf{E}\left|\langle u_\varepsilon(r)^3, p_{t-r}*\phi\rangle\right|^2
\lesssim_{\tau,T,\phi}1,$$
and therefore
$$ \mathbf{E}\left|\lambda\varepsilon^{d-2}\int_s^t \langle u_\varepsilon(r)^3,p_{t-r}*\phi\rangle \dd r \right|^2
\lesssim_{\tau,T,\phi,\lambda}|t-s|^2.$$
Combining the two bounds gives
$$ \mathbf{E}|\langle u_\varepsilon(t)-u_\varepsilon(s),\phi\rangle|^2
\lesssim_{\tau,T,\phi,\lambda}|t-s|^{2\alpha}$$
for some $\alpha>1/2$, uniformly in $\varepsilon$. Thus, by Kolmogorov's criterion, the family 
$(\langle u_\varepsilon(t),\phi\rangle)_{t\in[\tau,T]}$ is tight in $C([\tau,T])$. Since $\tau,T$ 
and $\phi$ were arbitrary, Mitoma's criterion implies tightness of $(u_\varepsilon)_{\varepsilon>0}$ 
in $C((0,\infty),\mathcal S'(\mathbb R^d))$.

\begin{corollary}
  Let $d\geq 3$ and defining $u_\eps$ as in \eqref{eq:rescaled_u}, the family $(u_\eps)_{\eps > 0}$ converges in law in 
  $C\left((0, \infty), \mathcal{S}'(\mathbb{R}^d)\right)$ to $(\sigma_\lambda p_t * \tilde{\xi})_{t>0}$ where 
  $\tilde{\xi}$ is a spatial white noise and $\sigma_\lambda$ is a positive constant. 
\end{corollary}

\subsection{Proof of non-triviality and creation of noise}\label{subsec:non_triv-proof}
Let us now prove that the variance is non-trivial and that noise is created, that is, items \ref{item:nontriv} and \ref{item:creation} of Theorem \ref{thm:Main}.
\begin{proof}[Proof of \ref{item:nontriv}]
In order to prove this, it suffices to show that for any fixed $t>0$, one has that
\(\mathbf{E}[u_\varepsilon(t,x)^2] \to 0\) as $\lambda \to +\infty$ uniformly in $\varepsilon$.

Note that the naïve application of Lemma \ref{lem:coming-down} only yields a bound of order 
$\varepsilon^{2-d}$, which is not sufficient. As explained in the outline of the proof, the damping of 
the ODE dominates only at very short times, hence yielding a very poor control in the limit.
Nonetheless, leveraging on this remark, we can still apply this estimate for short times, and then use the Malliavin 
calculus techniques to obtain a better bound for larger times. We provide this idea in full detail in Lemma \ref{lem:upgraded-cd}
from which one immediately obtains this result as a corollary.
\end{proof}

\begin{lemma}\label{lem:upgraded-cd}
  One has that for any \(t\geq \varepsilon^2\), and \( \theta \in [0,\frac{1}{2})\) ,
  \[\mathbf{E}[u_\varepsilon(t,0)^2] \lesssim_\theta \lambda^{-\theta} t^{-d/2}.\]
\end{lemma}
\begin{proof}
By projecting onto the first chaos (recall Equation~\eqref{eq:malliavin-lower}), 
we have by Lemma \ref{lem:coming-down} that 
\begin{equation}\label{eq:upgraded-cd-1}
  \int \mathbf{E}[D_z u_\varepsilon(\varepsilon^2, 0)]^2 \dd z \le \mathbf{E}[u_\varepsilon(\varepsilon^2, 0)^2] 
    \lesssim \lambda^{-1} \varepsilon^{-d}.
\end{equation}
Thus, by recalling Corollary~\ref{cor:max_principle}, we have that 
\(0 < D_z u_\varepsilon(\varepsilon^2, 0) \le (p_{\varepsilon^2} * \rho_\eps)(z)\) and so,
\begin{align*}
  \mathbf{E}[\|D u_\eps(\eps^2, 0)\|_\CH^2] 
  & = \mathbf{E}\left[\int (D_z u_\eps(\eps^2, 0))^2 \dd z\right]
    \le \int \mathbf{E}[D_z u_\eps(\eps^2, 0)] (p_{\eps^2} * \rho_\eps)(\dd z)\\
  & \lesssim \left(\int \mathbf{E}[D_z u_\eps(\eps^2,0)]^2 (p_{\eps^2} * \rho_\eps)(\dd z)\right)^{\frac{1}{2}}\\
  & \lesssim \eps^{-\frac{d}{2}}\left(\int \mathbf{E}[D_z u_\eps(\eps^2, 0)]^2 \dd z\right)^{\frac{1}{2}} \lesssim \lambda^{-1/2}\varepsilon^{-d}
\end{align*}
where we used Equation~\eqref{eq:upgraded-cd-1} in the last step. Hence, by Cauchy-Schwarz, we obtain that 
\[\int \mathbf{E}[D_z u_\eps(\eps^2, x)D_z u_\eps(\eps^2, y)] \dd z \le \mathbf{E}[\|D u_\eps(\eps^2, 0)\|_\CH^2] 
\lesssim \lambda^{-1/2}\varepsilon^{-d}\]
for any \(x, y \in \mathbb{R}^d\). Thus, interpolating this inequality 
with the na\"ive estimate obtained from Corollary~\ref{cor:max_principle}
we get that for any \(x, y\in \mathbb{R}^d\) and \(\theta \in [0, \frac{1}{2})\),
\[\int \mathbf{E}[D_z u_\eps(\eps^2, x)D_z u_\eps(\eps^2, y)] \dd z \lesssim \lambda^{-\theta} p_{c(\theta)\varepsilon^2} (x - y)\]
where we used the fact that \(\rho\) is compactly supported and so, 
\(p_{\varepsilon^2} * \rho_\eps \lesssim p_{2\varepsilon^2}\).
Finally, by the comparison principle, we have that
\(D_z u_\eps(t, \cdot) \le p_{t - \eps^2} * D_z u_\eps(\eps^2, \cdot)\) for any \(t \ge \eps^2\), and thus,
\begin{align*}
  \int \mathbf{E}[D_z u_\eps(t, 0)^2] \dd z 
  & \le \int \int \int \mathbf{E}[D_z u_\eps(\eps^2, x)D_z u_\eps(\eps^2, y)] \dd z p_{t - \eps^2}(x)p_{t - \eps^2}(y) \dd x \dd y\\
  & \lesssim \lambda^{-\theta} \int \int p_{c(\theta)}(x - y) p_{t - \eps^2}(x)p_{t - \eps^2}(y) \dd x \dd y\\
  & = \lambda^{-\theta} p_{2(t - \eps^2) + c(\theta)\varepsilon^2}(0) \lesssim \lambda^{-\theta} t^{-d/2}.
\end{align*}
Consequently, we obtain the desired bound by the Poincar\'e inequality.
\end{proof}
Finally, we conclude this section with the proof of the creation of noise, i.e. item \ref{item:creation} of Theorem \ref{thm:Main}.
\begin{proof}[Proof of \ref{item:creation}]
The joint Gaussianity follows directly by applying the Central Limit Theorem for spatial averages to linear combinations 
of \(u_\varepsilon(t,\cdot)\) and \(p_t*u_\varepsilon(0,\cdot)\). Moreover, the correlation of the limiting 
objects are non-zero as 
\[\mathbf{E}[u_\varepsilon(t,x)p_t*u_\varepsilon(0,x)] 
  = \mathbf{E}[\Pi_1 u_\varepsilon(t,x)p_t*u_\varepsilon(0,x) ] \gtrsim \|p_t*\rho_\varepsilon\|_2^2. \]
Thus, it remains to show that they are not fully correlated. For this, thanks to the orthogonality of 
chaoses, it suffices to show that \(\mathbf{E}[|\Pi_{3}u_\varepsilon(t,x)|^2]\) is lower bounded 
uniformly in $\eps$, say for $t \geq \eps^2$. Indeed, by the joint convergence of 
\((u_\varepsilon(t,x), p_t*u_\varepsilon(0,x))\) (cf. Remark~\ref{rmk:pointwise}),
we have that
\begin{align*}
  \sigma^2_\lambda c^2_\lambda \|p_t\|_{2}^4 & = \lim_{\varepsilon \to 0} \mathbf{E}[u_\varepsilon(t,x)p_t*u_\varepsilon(0,x)]^2 
   = \lim_{\varepsilon \to 0} \mathbf{E}[\Pi_1 u_\varepsilon(t,x) p_t*u_\varepsilon(0,x)]^2 \\
  & \le \limsup_{\varepsilon \to 0} \mathbf{E}[(\Pi_1 u_\varepsilon(t,x))^2] \mathbf{E}[(p_t*u_\varepsilon(0,x))^2]\\
  & = \limsup_{\varepsilon \to 0} (\mathbf{E}[u_\varepsilon(t,x)^2] - \mathbf{E}[(\Pi_3 u_\varepsilon(t,x))^2]) \mathbf{E}[(p_t*u_\varepsilon(0,x))^2]\\
  & = (\sigma_\lambda^2 - \liminf_{\varepsilon \to 0} \mathbf{E}[|\Pi_3 u_\varepsilon(t,x)|^2]) \|p_t\|_{2}^4.
\end{align*}
Consequently, we have that \(c_\lambda < 1\) as long as \(\liminf_{\varepsilon \to 0} \mathbf{E}[|\Pi_3 u_\varepsilon(t,x)|^2] > 0\).

For this, we recall the mild formulation of the equation satisfied by \(u\)
\begin{align*}
  u_\eps(t,x) & = p_t * u_\eps(0, x) - \lambda_\eps \int_0^t p_{t-s} * u_\eps(s, x)^3 \dd s \\
    & =: X_\eps(t,x) - \lambda \mathcal N_\eps(u_\eps,u_\eps,u_\eps)(t,x)
\end{align*}
where we introduced the trilinear operator \(\mathcal{N}_\eps\) by defining
\[\mathcal{N}_\eps(f, g, h) (t,x) := \eps^{d - 2} \int_0^t p_{t-s} * (f(s, \cdot) g(s, \cdot) h(s, \cdot))(x) \dd s.\]
Writing the first Picard iteration by setting $Y_\eps(t,x) := \mathcal N_\eps(u_\eps,u_\eps,u_\eps)(t,x)$, 
we have that
\begin{align*}
 u_\eps(t,x) &= X_\eps(t,x) - \lambda \mathcal{N}_\eps(X_\eps,X_\eps,X_\eps)(t,x) + 3\lambda^2 \mathcal{N}_\eps(X_\eps,X_\eps,Y_\eps)(t,x) \\
 & - 3\lambda^3\mathcal{N}_\eps(X_\eps,Y_\eps,Y_\eps)(t,x) + \lambda^4 \mathcal{N}_\eps(Y_\eps,Y_\eps,Y_\eps)(t,x).
\end{align*}
Thus, by noting that (via a back of the envelope computation) all terms involving \(\mathcal{N}_\eps\) 
are expected to be of the same order, one anticipates, and proves thereafter, that for sufficiently small $\lambda$,
one has that
\[u_\eps(t,x) = X_\eps(t,x) - \lambda \mathcal{N}_\eps(X_\eps,X_\eps,X_\eps)(t,x) + \lambda^2 R_\eps(t,x)\]
with \(\mathbf{E}[R_\eps(t,x)^2] \lesssim t^{-\frac{d}{2}}\) uniformly in \(\eps\) and \(\lambda\).
Consequently, assuming the above analysis, we obtain the desired lower bound provided that 
\(\mathbf{E}[|\Pi_3 \mathcal{N}_\eps(X_\eps,X_\eps,X_\eps)(t,x)|^2] \gtrsim t^{-\frac{d}{2}}\) uniformly 
in \(\eps\). This lower bound follows as
\begin{align*}
&\mathbf{E}[|\Pi_3 \mathcal{N}_\eps(X_\eps,X_\eps,X_\eps)(t,x)|^2]\\
=\ & \eps^{2d-4}\int_0^t \int_0^t \int \int p_{t-s}(x-y) p_{t-s'}(x-y') \mathbf{E}[\Pi_3 (X_\eps(s,y)^3) \Pi_3 (X_\eps(s',y')^3)] \dd y \dd y' \dd s \dd s'\\
=\ & 6\eps^{2d-4} \int_0^t \int_0^t p_{2t - s - s'} * (p_{s+s'} * \rho_\eps * \tilde \rho_\eps)^3(0) \dd s \dd s' \\
\gtrsim\ & \eps^{2d-4}t^{-\frac{d}{2}} \int_{\tfrac{\eps^2}{2}}^t \int_{\tfrac{\eps^2}{2}}^t \frac{1}{\left(s+s'\right)^d} \dd s \dd s' \asymp t^{-\frac{d}{2}}
\end{align*}
where we used the following equality for the second equality
\begin{align*}
  \mathbf{E}[\Pi_3 (X_\eps(s, y)^3) \Pi_3 (X_\eps(s', y')^3)] 
  & = 6 \mathbf{E}[X_\eps(s,y)X_\eps(s', y')]^3\\ 
  & = 6 (p_{s+s'} * \rho_\eps * \tilde \rho_\eps)^3(y - y'),
\end{align*}
and where for the inequality we restricted the integration domain and used that thanks to the compact support of $\rho$,
\[ p_r * \rho_\eps *\tilde{\rho}_\eps \gtrsim p_r \text{ pointwise whenever } r \gtrsim \varepsilon^2.\]
The last step follows by a direct computation, and we obtain the desired lower bound on the 
asymptotic. 

We now complete the proof by arguing that \(\mathbf{E}[R_\eps(t,x)^2] \lesssim t^{-\frac{d}{2}}\) uniformly in 
\(\eps\) and \(\lambda\).
For this, we estimate independently each of the 3 terms involving the non-linearity $\mathcal{N}_\eps$ in the 
definition of \(R_\eps\). For instance, for the term \(\mathcal{N}_\eps(X_\eps,X_\eps,Y_\eps)\), we have that
\begin{align*}
&\mathbf{E}[|\mathcal{N}_\eps(X_\eps, X_\eps, Y_\eps)(t,x)|^2]\\
& = \int_0^t \int_0^t \int \int p_{t-s}(x-y) p_{t-s'}(x-y') I_\varepsilon^{(2, 1)}(s, s', y, y') \dd y \dd y' \dd s \dd s'.
\end{align*}
where we have set 
\[I_\varepsilon^{(n, m)}(s, s', y, y') := 
  \varepsilon^{2(d - 2)}\mathbf{E}[X_\eps^n(s, y)X_\eps^m(s', y')Y_\eps^n(s, y)Y_\eps^m(s', y')]\]
Similarly, the variance for the other two terms have the same structure where the only difference is 
the value of $n$ and $m$ in the definition of $I_\varepsilon^{(n, m)}$ with $n + m = 3$. Let us now estimate $I_\varepsilon^{(n, m)}$ 
directly. As $X^n_\eps Y^m_\eps$ is centered, by the covariance inequality, we have that
{\hfuzz=15pt
\begin{align*}
  |I_\varepsilon^{(n, m)}(s, s', y, y')| & \le 
    \varepsilon^{2(d - 2)} \int \sqrt{\mathbf{E}[|D_z(X^n_\eps Y^m_\eps (s, y))|^2]} 
      \sqrt{\mathbf{E}[|D_z(X^n_\eps Y^m_\eps(s', y'))|^2]} \dd z,
\end{align*}}
and we focus on estimating terms of the form \(|D_z(X^n_\eps Y^m_\eps (s, y))|\). 
To this end, by the product rule for the Malliavin derivative, it suffices if we can provide an 
estimate for the terms of the form
\begin{equation}\label{eq:DXDY-est}
  \begin{rcases}
    & \mathbf{E}[|X_\eps^{n - 1} Y_\eps^m D_z X_\eps(s, y)|^2]\\
    & \mathbf{E}[|X_\eps^{n} Y_\eps^{m - 1} D_z Y_\eps(s, y)|^2]
  \end{rcases}\lesssim (s+\varepsilon^2)^{-d} p_{s+\varepsilon^2}^2(y - z).
\end{equation}
Indeed, assuming \eqref{eq:DXDY-est}, it holds that
\begin{align*}
  |I_\varepsilon^{(n, m)}(s, s', y, y')| & \le 
    \varepsilon^{2(d - 2)} (s+\varepsilon^2)^{-\frac{d}{2}}(s'+\varepsilon^2)^{-\frac{d}{2}} p_{s+s'+\varepsilon^2}(y-y'),
\end{align*}
from which the desired upper bound on $R_\varepsilon$ is immediate.

The computation required for \eqref{eq:DXDY-est} is rather explicit and we collect below the 
inequalities we need to complete the recipe. Straightaway, we compute 
\begin{equation}\label{eq:DX-est}
  D_z X_\eps(s, y) = (p_s * \rho_\eps)(z - y) \lesssim p_{s+\varepsilon^2}(z - y)
\end{equation}
and
\begin{equation}\label{eq:DY-est}
  \begin{aligned}
    D_z Y_\eps(s, y) & = \eps^{d - 2} \int _0^s (p_{s - r} * D_z (u_\eps(r, \cdot)^3))(y)\dd r\\
    & = \eps^{d - 2} \int _0^s \left(p_{s - r} * (3(u_\eps(r, \cdot)^2D_z u_\eps(r, \cdot)))\right)(y)\dd r \\
    & \lesssim \varepsilon^{d - 2} \int_0^s (p_{s - r} * (u_\eps(r, \cdot)^2 (p_{r+\eps^2})(z - \cdot)))(y) \dd r
  \end{aligned}
\end{equation}
where we used the point-wise bound on the Malliavin derivative of $u$ (i.e. Corollary~\ref{cor:max_principle}).
Note that this also shows that $D_zY_\varepsilon$ is pointwise non-negative, so that the bound actually holds in absolute value.

Moreover, we have that 
\begin{equation}\label{eq:X4-est}
  \mathbf{E}[|X_\eps(s, y)|^8] = \mathbf{E}[|(p_s * \rho_\eps * \xi)(y)|^8] \sim (s+\varepsilon^2)^{-2d}
\end{equation}
and by Equation~\eqref{eq:DY-est} and Minkowski's and the Cauchy-Schwarz inequality, we have that
\begin{align*}
    &\mathbf{E}[|Y_\eps(s, y)|^8] \lesssim \mathbf{E}[\|DY_\eps(s, y)\|_{\CH}^8] 
      \le \left(\int \mathbf{E}[|D_z Y_\eps(s, y)|^8]^{\frac{1}{4}} \dd z\right)^4\\
    \lesssim\ &\varepsilon^{8(d - 2)} \left(\int\mathbf{E}^{\frac{1}{4}}\left[\left(\int_0^s (p_{s - r} * (u_\eps(r, \cdot)^2 (p_{r+\eps^2})(z - \cdot)))(y) \dd r\right)^8\right] \dd z\right)^4\\
    \le\ &\varepsilon^{8(d - 2)} \left(\int \left(\int_0^s (p_{s - r} * (\mathbf{E}[u_\eps(r, \cdot)^{16}]^{\frac{1}{8}} (p_{r+\eps^2})(z - \cdot)))(y) \dd r\right)^2 \dd z\right)^4\\
    \lesssim\ &\varepsilon^{8(d - 2)} \left(\int \left(\int_0^s \left(r+\eps^2\right)^{-\tfrac{d}{2}}(p_{s - r} * p_{r + \eps^2}(z - \cdot))(y) \dd r\right)^2 \dd z\right)^4.
\end{align*}
where we used the fact that \(\rho\) is sub-Gaussian and the Poincaré bounds on the $16^{th}$ moment of $u$ in the last step.
We may now use the Chapman-Kolmogorov identity and compute the inner time integral, which yields as claimed

\begin{equation}\label{eq:Y4-est}
  \mathbf{E}[|Y_\eps(s, y)|^8] \lesssim \mathbf{E}[\|DY_\eps(s, y)\|_{\CH}^8] \lesssim \left(s+\eps^2\right)^{-2d}.
\end{equation}
Similarly, by the same argument as above, we find 
\begin{align}\label{eq:DY2-est}
  \mathbf{E}[|D_z Y_\eps(s, y)|^4] &\lesssim \varepsilon^{4(d - 2)} 
    \left(\int_0^s (p_{s - r} * (p_{r + \eps^2}(\cdot)p_{r + \eps^2}(z - \cdot)))(y) \dd r\right)^4\\
    &\lesssim p_{s+\varepsilon^2}^4(y - z).
\end{align}
Thus, by combining the above estimates with repeated applications of the Cauchy-Schwarz inequality, we conclude 
\eqref{eq:DXDY-est} as desired.

\end{proof}
\begin{remark}\label{rem:tricorr2d}
Let us compare the arguments used above with the situation in dimension $2$. This is best seen when comparing with the proof mechanism in \cite{gabrielrosati}.
If one performs again the expansion
\[ u_\varepsilon(t,x) = X_\varepsilon(t,x) - \lambda \mathcal{N}_\varepsilon(X_\varepsilon,X_\varepsilon,X_\varepsilon)(t,x) + \cdots,\]
it can be noted from the same computations as above that now, for $t$ of order $1$, one has that
\[ \mathbf{E}[X_\varepsilon(t,x)^2] \asymp 1\] 
which is analogous to the case of higher dimensions. However, a careful investigation of the third order term yields 
\[\mathbf{E}[|\Pi_1 \mathcal{N}_\varepsilon(X_\varepsilon,X_\varepsilon,X_\varepsilon)(t,x)|^2] \sim (\log \tfrac{1}{\varepsilon})^2 \]
on the other hand, one has
\[\mathbf{E}[|\Pi_3 \mathcal{N}_\varepsilon(X_\varepsilon,X_\varepsilon,X_\varepsilon)(t,x)|^2] \sim \log \tfrac{1}{\varepsilon}.\]
While this divergence motivates the choice of the weak coupling\footnote{Note that here the choice of notation differs from the choice in \cite{gabrielrosati}, but corresponds nonetheless to the same weak coupling regime} $\lambda = \lambda_\eps \asymp (\log \tfrac{1}{\eps})^{-1}$ and contrasts strongly with $d\geq 3$ where these quantities remain bounded, this was to be expected because of the criticality in $d=2$.
However, it may be more surprising that they do so at different orders, and this is precisely what prevents higher chaoses from contributing to the limiting noise.
Indeed, in the example above, the two terms have the same homogeneity in $\lambda_\varepsilon$, so that taming the first divergence necessarily kills the second term completely. 
This explains the triviality of the correlations in $d=2$. On the other hand, it is also clear from this picture that a non-trivial dependence in $\lambda$ of the variance is still to be expected in the weak coupling regime.

\end{remark}

\begin{appendix}

\section{Central Limit Theorem for short-range correlated stationary fields}\label{app:clt}
In this appendix, we state a Central Limit Theorem for short-range correlated fields. Similar results are classical but we provide a short proof under assumptions adapted to our setting for completeness.

\begin{theorem}\label{thm:clt}
  Let $\eta$ be a smooth random field on $\mathbb{R}^d$ which is stationary, centered, with moments of any order, odd mixed moments vanishing, and with the following exponential decay of correlations
   \begin{align}
  &\left|\mathbf{E}\left[\prod_{i=1}^p \eta(x_i) \prod_{k=1}^q \eta(y_k)\right]-\mathbf{E}\left[\prod_{i=1}^p \eta(x_i)\right]\mathbf{E}\left[\prod_{k=1}^q \eta(y_k)\right]\right|\\
  \lesssim_{p,q}\ & \exp\left(-C_{p,q}\underset{i,k}{\inf}\|x_i - y_k\|\right). \end{align}
  Then, \(\eta_\eps(x) := \eps^{-\frac{d}{2}} \eta\left(\frac{x}{\eps}\right)\)
  satisfies that for any $m \geq 0, \Phi: \mathbb{R}^d \rightarrow \mathbb{R}^m$ smooth and fast decaying, \[\langle \eta_\eps, \Phi \rangle := \int_{\mathbb{R}^d} \eta_\eps(x) \Phi(x) \dd x \underset{\eps\to 0}{\Longrightarrow} \mathcal{N}\left(0, \Sigma^2(\Phi)\right)\]
  with
  \(\sigma^2 := \int_{\mathbb{R}^d} \mathbf{E}[\eta(0) \eta(x)] \dd x \) and \( \Sigma_{i,j}^2(\Phi) := \sigma^2 \int_{\mathbb{R}^d} \Phi_i(x) \Phi_j(x) \dd x.\)
\end{theorem}
\begin{proof}
By the Cramer-Wold theorem, it is enough to prove the result in the case of a single test function
$\varphi\in \mathcal S(\mathbb R^d)$. Indeed, applying the scalar result to linear combinations $\varphi=a_1\Phi_1+\cdots+a_m\Phi_m$ identifies the covariance matrix of the limiting vector. We therefore prove that
$$\langle \eta_\varepsilon,\varphi\rangle = \int_{\mathbb R^d}\varepsilon^{-d/2}\eta(x/\varepsilon)\varphi(x)\dd x$$
converges in law to a centered Gaussian random variable with variance
$$ \sigma^2\|\varphi\|_{L^2(\mathbb R^d)}^2, \qquad
\sigma^2:=\int_{\mathbb R^d}\mathbf E[\eta(0)\eta(x)]\dd x.$$
The integrability of $x\mapsto\mathbf E[\eta(0)\eta(x)]$ follows from the decorrelation assumption with $p=q=1$, since $\eta$ is centered.

We prove convergence of moments. All odd moments vanish by assumption, so it remains to identify the even moments. Fix $k\ge 1$. We have
$$ \mathbf E[\langle \eta_\varepsilon,\varphi\rangle^{2k}] = \varepsilon^{-kd}
\int_{(\mathbb R^d)^{2k}} \mathbf E\left[\prod_{i=1}^{2k}\eta(x_i/\varepsilon)\right]
\prod_{i=1}^{2k}\varphi(x_i)\,dx_1\cdots \dd x_{2k}.$$
Set
$$ R_\varepsilon:=|\log\varepsilon|^2, \qquad
r_\varepsilon:=\varepsilon R_\varepsilon.$$
Given a point configuration $x=(x_1,\ldots,x_{2k})$, define a partition $\mathcal P_\varepsilon(x)$ of $\{1,\ldots,2k\}$ by taking the
equivalence relation generated by closing the partial relation
$$ i\sim j \quad\Longleftrightarrow\quad |x_i-x_j|\le r_\varepsilon.$$
Under this construction, if $B\neq B'$ are two distinct blocks of $\mathcal P_\varepsilon(x)$, then
$$\inf_{i\in B,j\in B'}|x_i-x_j|>r_\varepsilon.$$
In the variables $\frac{x}{\varepsilon}$ this gives separation at least $R_\varepsilon$ among blocks.

For a fixed partition $\mathcal P$ of $\{1,\ldots,2k\}$, denote by $I_\varepsilon(\mathcal P)$ the contribution of the region
$\{\mathcal P_\varepsilon(x)=\mathcal P\}$ to the above moment. On this region, the decorrelation assumption applied recursively to the different
blocks gives
$$\left|\mathbf E\left[\prod_{i=1}^{2k}\eta(x_i/\varepsilon)\right] -
\prod_{B\in\mathcal P} \mathbf E\left[\prod_{i\in B}\eta(x_i/\varepsilon)\right]\right|
\lesssim_k e^{-c_k R_\varepsilon}.$$
Indeed, distinct blocks are separated by at least $R_\varepsilon$ in the microscopic variables. Consequently, the total contribution of this error
is bounded by
$$\varepsilon^{-kd} e^{-c_k|\log\varepsilon|^2}
\prod_{i=1}^{2k}\|\varphi\|_{L^1(\mathbb R^d)}
\longrightarrow 0,$$
because $e^{-c_k|\log\varepsilon|^2}$ decays faster than any power of $\varepsilon$. Consequently, the moments can be factorized into the blocks 
of $\mathcal{P}$ under $\mathcal{P}_\varepsilon(x) = \mathcal{P}$.

It remains to study the factorized block contributions. If a block $B\in\mathcal P$ has odd cardinality, then
$$\mathbf E\left[\prod_{i\in B}\eta(x_i/\varepsilon)\right]=0$$
by the assumption that all odd moments vanish. Hence only partitions all of whose blocks have even cardinality can contribute.

To prove Gaussianity, we then need to show that any block of size at least four gives a vanishing
contribution. Let $B$ be a block with $m:=|B|\ge4$, and choose a representative $i(B)\in B$. On the region
$\{\mathcal P_\varepsilon(x)=\mathcal P\}$, all points in $B$ are connected to $x_{i(B)}$ by a chain of at most $m-1$ edges of length at
most $r_\varepsilon$. Hence
$$|x_i-x_{i(B)}|\le (m-1)r_\varepsilon, \qquad i\in B.$$
Using the uniform moment bound
$$ \left|\mathbf E\left[\prod_{i\in B}\eta(x_i/\varepsilon)\right]\right| \le C_m,$$
we may estimate the contribution of the variables belonging to $B$ by integrating first with respect to the representative point
$x_{i(B)}$ and then with respect to the remaining $m-1$ points in the block. Once $x_{i(B)}$ is fixed, each of the other points $x_i$,
$i\in B\setminus\{i(B)\}$, is constrained to lie in a ball of radius $O(r_\varepsilon)$ around $x_{i(B)}$. Therefore the available volume for
these relative coordinates is at most of order $r_\varepsilon^{d(m-1)}$. Since $\varphi$ is a bounded function, the
product $\prod_{i\in B}|\varphi(x_i)|$ is bounded by a multiple of $|\phi(x_{i(B)})|$,
hence an integrable function of $x_{i(B)}$, uniformly in the relative coordinates. It follows that the
integral over the variables in $B$ is bounded by a constant depending only on $\varphi$ and $m$, multiplied by the volume factor
$r_\varepsilon^{d(m-1)}$. Combining this with the prefactor $\varepsilon^{-dm/2}$, we get that the absolute value of the corresponding
block integral is bounded by
$$ C_{\varphi,m}\ \varepsilon^{-dm/2} r_\varepsilon^{d(m-1)} = C_{\varphi,m}\ \varepsilon^{d(m/2-1)}
|\log\varepsilon|^{2d(m-1)}.$$
Since $m\ge4$, this tends to zero as $\varepsilon\to0$. Pair blocks, on the other hand, are uniformly bounded. Therefore any partition containing
a block of size at least four has vanishing contribution.

From this one is left to compute the contribution of pairings. Let $\mathcal P$ be a pairing of
$\{1,\ldots,2k\}$. Write its blocks as
$$\mathcal P=\{\{a_1,b_1\},\ldots,\{a_k,b_k\}\}.$$
For each pair, set
$$y_\ell=\frac{x_{b_\ell}-x_{a_\ell}}{\varepsilon}, \qquad \ell=1,\ldots,k.$$
After the change of variables $x_{b_\ell}=x_{a_\ell}+\varepsilon y_\ell$, the contribution of the pairing $\mathcal P$ is
$$\int \mathbf 1_{\{\mathcal P_\varepsilon(x)=\mathcal P\}} \prod_{\ell=1}^k
\mathbf E[\eta(0)\eta(y_\ell)] \prod_{\ell=1}^k \varphi(x_{a_\ell}) \varphi(x_{a_\ell}+\varepsilon y_\ell)
\dd x_{a_1}\cdots \dd x_{a_k}\dd y_1\cdots \dd y_k +o(1).$$
Here the condition $\mathcal P_\varepsilon(x)=\mathcal P$ includes $|y_\ell|\le R_\varepsilon$ for each $\ell$, together with separation
conditions between distinct pairs. Since $R_\varepsilon\to\infty$ and $r_\varepsilon\to0$, the indicator
$\mathbf 1_{\{\mathcal P_\varepsilon(x)=\mathcal P\}}$ converges pointwise almost everywhere to $1$. Moreover,
$$\left|\mathbf E[\eta(0)\eta(y_\ell)]\right|\lesssim e^{-c|y_\ell|}$$
and $\varphi$ is rapidly decaying. Hence dominated convergence gives
$$ I_\varepsilon(\mathcal P) \to \prod_{\ell=1}^k
\left( \int_{\mathbb R^d}\mathbf E[\eta(0)\eta(y_\ell)]\dd y_\ell \right)
\prod_{\ell=1}^k \left( \int_{\mathbb R^d}\varphi(x_{a_\ell})^2\dd x_{a_\ell} \right).$$
Therefore
$$ I_\varepsilon(\mathcal P) \to \sigma^{2k}\|\varphi\|_{L^2(\mathbb R^d)}^{2k}.$$
Consequently, by summing over pairings one obtains
$$ \mathbf E[\langle \eta_\varepsilon,\varphi\rangle^{2k}]\to \frac{(2k)!}{2^k k!}
\sigma^{2k}\|\varphi\|_{L^2(\mathbb R^d)}^{2k}.$$
Together with the vanishing of all odd moments, this identifies the limiting moments as that of the centered Gaussian distribution with variance
$\sigma^2\|\varphi\|_{L^2(\mathbb R^d)}^2$. Since the Gaussian
distribution is characterized by its moments, the desired convergence in law follows.

Finally, applying the scalar result to all linear combinations
$a_1\Phi_1+\cdots+a_m\Phi_m$ and using the Cramer-Wold theorem yields
the stated vector-valued convergence with covariance matrix
$$ \Sigma^2_{ij}(\Phi) = \sigma^2 \int_{\mathbb R^d}\Phi_i(x)\Phi_j(x)\dd x,$$
which concludes the proof.
\end{proof}

\begin{remark}
In the proof of the main result Theorem~\ref{thm:Main}, we actually use the claim for $\varepsilon$ dependent test functions $\Phi = \Phi^{(\varepsilon)}$.
The proof above extends clearly to this case as long as $(\Phi^{(\varepsilon)})_{\varepsilon > 0}$ converges pointwise and is dominated, which is enough for our sake.

\end{remark}
  
\section{Deterministic preliminaries on the Allen-Cahn equation}\label{app:det}
In this appendix we recall for completeness some standard results on  the deterministic solution theory, required to derive rigorously the mild formulation and the PDE for the Malliavin derivative of the solution.
Since we work in the whole space, the initial data of interest are unbounded and hence we need to work with functions of controlled growth at infinity.

In the following, we let $a>0$ be fixed, and introduce the notation
\[ \| f \|_{a} := \sup_{x \in \mathbb{R}^d} |f(x)| e^{-a|x|}\]
for any \(f \in C^0(\mathbb{R}^d)\) (i.e. any continuous function \(f : \mathbb{R}^d \to \mathbb{R}\)).
Classically, 
\[X_a := \left\{  f \in C^0(\mathbb{R}^d) : \sup_{x \in \mathbb{R}^d} |f(x)| e^{-a|x|} < +\infty.\right\}\]
equipped with the above norm is a Banach space. In this setting, let us prove the following.
\begin{lemma}\label{lemma:GWP}
For any $u_0 \in X_a$,
the Equation~\eqref{eq:AC} 
\[\partial_t u(t,x) = \Delta u(t,x) - \lambda u(t,x) ^3,\quad t>0,\:x\in \mathbb{R}^d,\]
started from the initial data $u_0$, has a unique classical solution. Furthermore, one has that 
\[\|u(t, \cdot)\|_a = \sup_{x \in \mathbb{R}^d} |u(t,x)| e^{-a|x|} \lesssim_t \|u_0\|_a.\]  
\end{lemma}
\begin{proof}
Let us start with existence. We cannot perform directly a fixed point formulation for this, as the non-linearity is superlinear, and thus
the growth at infinity cannot be controlled in the direct Duhamel formulation.

To fix this issue, let us consider the following equation with truncated non-linearity, for $R>0$.

\[\partial_t u_R(t,x) = \Delta u_R(t,x) - \lambda \left(u_R(t,x)^2\wedge R\right) u_R(t,x),\quad t>0,\:x\in \mathbb{R}^d.\]

Local solutions can be obtained via a fixed-point argument.
For a fixed $R$, the map 

\begin{equation*}
\Psi_{R,T}:
\begin{cases}  &C([0,T],X_a)  \to  C([0,T],X_a)\\
               & \qquad \quad v \qquad \mapsto  p_t*u_0 -3\lambda \int_0^t p_{t-s}*\left(\left(v_s^2\wedge R \right) v_s\right) \dd s\\
\end{cases}
\end{equation*}
is a contraction for small enough $T$. Indeed, we have the following growth estimate for the heat kernel
\[ \| p_t * f \|_a \lesssim e^{a\tfrac{t^2}{2}} \|f\|_a \text{, since } \int_{\mathbb{R}^d} e^{a|y|} p_t(x-y) \dd y \lesssim e^{a\tfrac{t^2}{2}}.\]
It follows that
\[ \|\Psi_{R,T}(v)_t\|_a \lesssim e^{a\tfrac{T^2}{2}}\|u_0\|_a + RT e^{a\tfrac{T^2}{2}}\|v\|_a \]
hence $\Psi_{R,T}$ is well-defined, and since $x \mapsto (x^2\wedge R) x$ is $3R$-Lipschitz,
\[ \|\Psi_{R,T}(v)_t -\Psi_{R,T}(w)_t \|_a \lesssim RT e^{a\tfrac{T^2}{2}}\|v-w\|_a .\]
Consequently, for $T>0$ small enough, $\Psi_{R,T}$ is a contraction.

By the maximum principle, one has a priori that the fixed point $u_R$ actually satisfies $ \|u_R(t) \|_a \leq e^{a\tfrac{t^2}{2}} \|u_0\|_a $, so that
the solution is actually globally defined.

Now, by the following smoothing estimates in the $X_a$ space (that are proved similarly as the growth estimate above)
\[  \|\nabla p_t * f\|_a \lesssim_{a,T} t^{-\tfrac{1}{2}} \|f\|_a,\: \|p_t* f - p_s*f\|_a \lesssim_{a,T} |t-s|^{\alpha}(t\wedge s)^{-\alpha}\|f\|_a \text{ for } 0 < \alpha < 1,\]
together with the Duhamel formulation yields that $v_R(t) = u_{R}(t) - p_t*u_0$ is such that
\[ \sup_{0\leq t\leq T} \|\nabla v_R(t) \|_{3a} + \sup_{0\leq s < t \leq T} \frac{\|v_R(t)-v_R(s)\|_{3a}}{|t-s|^{\tfrac{1}{2}}}  \lesssim_T \|u_0\|_a .\]
Consequently, one can pass to the limit $R\to \infty$ along a subsequence by compactness to get that $u_R$ converges pointwise to a continuous function $u \in C([0,T],X_{3a})$.

By dominated convergence (in particular, using the growth estimate for the heat kernel), one can pass to the limit in the fixed-point formulation to get that
\[u_t = p_t*u_0 -3\lambda \int_0^t p_{t-s}*\left(u_s^3\right) \dd s,\]
so that we get existence of a (classical) solution.
The claimed upper bound follows from the construction because it holds uniformly in $R$,  and the uniqueness is immediate from standard arguments.
Indeed, the difference $w = u-v $ of candidate solutions $u,v$ solves
\[ \partial_t w = \Delta w - \lambda (u^2+uv+v^2) w,\qquad w(0,\cdot) = 0,\] 
hence vanishes. This completes the proof.

\end{proof}
We state here the following differentiability result for the solution map constructed above.
\begin{lemma}
For all $t>0,\:x\in \mathbb{R}^d,$  the map $J_{t,x} : X_a \to \mathbb{R}$ that sends $u_0$ to $u(t,x)$
is Fréchet differentiable. Furthermore, the Fréchet derivative at $u_0$ of the maps $(J_{t,x})_{t,x}$, defines for each $h \in X_a$
a function $v : t,x\mapsto v(t,x) := J'_{t,x} (u_0)[h] $ that is the unique classical solution of the equation
\[\partial_t v(t,x) = \Delta v(t,x) - 3 \lambda u(t,x) ^2v(t,x),\quad t>0,\:x\in \mathbb{R}^d,\]
with initial data $h$.
\end{lemma} 
\begin{proof}
To get the claimed result, we need to argue that for all $(t,x)$
\[J_{t,x}(u_0+h) = J_{t,x}(u_0) + v(t,x) + o\left(\|h\|_a\right)\]
where $v(t,x)$ is the unique solution of the equation given in the claim.
For this let us introduce the remainder $w(t,x) = J_{t,x}(u_0+h) - J_{t,x}(u_0) - v(t,x)$.
Let us also introduce the shorthand $J(u_0+h) = j_1, J(u_0)= j_2$.
By these definitions $w$ solves
\begin{align}
\partial_t w &= \Delta w - \lambda(j_1^3 -j_2^3-3 j_2^2v)  \\
&= \Delta w - \lambda\left((j_1^2+j_1j_2+j_2^2)(j_1-j_2) -3 j_2^2v\right)\\
&= \Delta w - \lambda(j_1^2+j_1j_2+j_2^2)w - \lambda(2 j_2^2 - j_1^2 - j_1j_2)v.
\end{align}
Since $j_1^2+j_1j_2+j_2^2 \geq 0$, this yields by the maximum principle and the heat-kernel growth estimates, that
\begin{align}
\|w(t)\|_{3a} &\lesssim_t \sup_{s\leq t} \|(j_2(s)-j_1(s))j_2(s)v(s)\|_{3a} + \|(j_2(s)-j_1(s))j_1(s)v(s)\|_{3a}\\
&\lesssim_t \|u_0\|_a\|h\|^2_a,
\end{align}
which concludes the proof of the claim.

\end{proof}
\end{appendix}

\printbibliography
\end{document}